\documentclass{article}
\usepackage{amssymb,amsmath,amsthm,graphicx,cite,arydshln}

\def\qed{\hfill {\hbox{${\vcenter{\vbox{               
   \hrule height 0.4pt\hbox{\vrule width 0.4pt height 6pt
   \kern5pt\vrule width 0.4pt}\hrule height 0.4pt}}}$}}}

\def\utr{\ \underline{\triangleright}\ }
\def\otr{\ \overline{\triangleright}\ }

\newtheorem{theorem}{Theorem}
\newtheorem{lemma}[theorem]{Lemma}
\newtheorem{proposition}[theorem]{Proposition}

\theoremstyle{definition}
\newtheorem{example}{Example}
\newtheorem{definition}{Definition}
\newtheorem{remark}{Remark}

\date{}

\title{\Large \textbf{Parity Biquandle Invariants of Virtual Knots}}
\author{
Aaron Kaestner\footnote{Email: amkaestner@northpark.edu‎}\and 
Sam Nelson\footnote{Email: sam.nelson@cmc.edu. Partially supported by Simons Foundation collaboration grant 316709}\and
Leo Selker\footnote{Email: lselker13@gmail.com }}

\begin{document}
\maketitle

\begin{abstract} We define counting and cocycle enhancement invariants of 
virtual knots using parity biquandles. The cocycle invariants are determined by 
pairs consisting of a biquandle
$2$-cocycle $\phi^0$ and a map $\phi^1$ with certain compatibility conditions
leading to one-variable or two-variable polynomial invariants of virtual knots.
We provide examples to show that the parity cocycle invariants can distinguish
virtual knots which are not distinguished by the corresponding non-parity 
invariants.
\end{abstract}

\medskip
\parbox{5.5in} {\textsc{Keywords:} Virtual knots, parity biquandles, cocycle 
invariants, enhancements of counting invariants

\smallskip

}

\section{\large\textbf{Introduction}}\label{I}

In 1996, Louis Kauffman introduced the world to \textit{virtual knot theory} 
in \cite{K}. Virtual knot theory is a combinatorial generalization of the
theory of knotted curves in $\mathbb{R}^3$, now known as \textit{classical 
knot theory}. Each ambient isotopy class of knotted oriented curves in 
$\mathbb{R}^3$ coincides with an equivalence class of combinatorial objects
known as \textit{signed Gauss diagrams}; however, the set of all such 
equivalence classes includes classes which do not correspond to classical 
knots. These extra classes are known as \textit{virtual knots}.
Virtual knots can be understood geometrically as knots in certain 
3-manifolds ($\Sigma\times [0,1]$ for $\Sigma$ an orientable surface) up to
equivalence by stabilization of $\Sigma$ \cite{CKS2}. 

In \cite{K}, it is observed that every classical knot is represented by a 
diagram in which every crossing is \textit{evenly intersticed}, i.e.\ every 
crossing has an even number of over-- and under--crossing points along the knot
between its over and under instances. In virtual knots, a crossing can have an 
even or odd number of crossing points between its over and under instances, 
and moreover this even or odd \textit{parity} is not changed by Reidemeister 
moves. In \cite{M}, parity was used to to create a number of invariants for 
virtual knots. In \cite{D} the notion of parity was generalized to 
integer-valued maps and used to define new invariants of virtual knots.
Parity invariants are very good at distinguishing classical knots from
non-classical virtual knots as well as simply distinguishing virtual knot types.

In \cite{FRS} (and see also \cite{KR}), algebraic structures known as
\textit{biquandles} were introduced. Given any finite biquandle $X$, there is a
non-negative integer-valued invariant of oriented knots and links known as 
the \textit{biquandle counting invariant} which counts homomorphisms from the 
fundamental biquandle of a knot $K$ to $X$, represented as colorings
of the semiarcs of $K$ by elements of $X$. Cocycles in the second cohomology
of a finite biquandle were first used to enhance the biquandle counting 
invariant in \cite{CES}.

In \cite{KK}, biquandles incorporating the notion of parity were introduced
(see also \cite{A}). 
In this paper we extend the counting invariant to the case of finite parity 
biquandles and define enhancements of the counting invariant using 
\textit{parity enhanced cocycles}, cocycles in the second cohomology of the
even part of the parity biquandle with extra information analogous to the
virtual cocycles in \cite{CN}.

The paper is organized as follows. In Section \ref{B} we review the basics of 
biquandles. In Section \ref{PB} we review parity biquandles and introduce the parity biquandle counting invariant. In Section \ref{E} we review biquandle 
cohomology and define parity cocycle enhancements of the counting invariant.
We provide examples demonstrating that the parity enhanced cocycle invariants 
are stronger than the corresponding unenhanced cocycle invariants and the 
corresponding non-parity invariants for virtual
knots. In Section \ref{Q} we conclude with questions for future research.

\section{\large\textbf{Biquandles}}\label{B}

A biquandle is an algebraic structure with axioms motivated by the Reidemeister
moves (see \cite{FRS,KR} etc.). It can be defined abstractly:
\begin{definition}
A \textit{biquandle} is a set $X$ along with two operators, $\otr$ and $ \utr$, 
both maps $X \times X \mapsto X \times X$, such that:
\begin{enumerate}
\item[(i)] For all $x \in X, x \otr x = x \utr x$
\item[(ii)] We have right invertibility of both maps and pairwise invertibility, i.e.\ the maps $\alpha_y:x \mapsto x\otr y, \beta_y:x \mapsto x\utr y$, and 
$S: (x,y) \mapsto (y\otr x, x\utr y)$ are all invertible.
\item[(iii)] For all $x,y,z\in X$, we have the \textit{exchange laws}: 
\begin{align*}
(z \otr y) \otr (x \otr y) &= (z \otr x) \otr (y \utr x)\\
(x \otr y) \utr (z \otr y) &= (x \utr z) \otr (y \utr z)\\
(y \utr x) \utr (z \otr x) &= (y \utr z) \utr (x \utr z)
\end{align*}
\end{enumerate}

\end{definition}

\begin{example}
A well-known type of biquandle is the \textit{Alexander biquandle}. The 
biquandle's underlying set $X$ is a module over the ring 
$\Lambda=\mathbb{Z}[t^{\pm 1}, s^{\pm 1}]$ of two-variable Laurent polynomials.
In particular, note that $s$ and $t$ are invertible, so for Alexander 
biquandles structures on finite rings or fields (where $s$ and $t$ are 
elements of the ring), the characteristic must be relatively 
prime to $s$ and $t$. The operations are defined as: 
\begin{align*}
x \utr y &= tx + (s^{-1}-t)y\\
x \otr y &= s^{-1}y
\end{align*}
The first biquandle axiom follows from the definition, the second follows from 
the fact that $s$ and $t$ are invertible, and the exchange laws can be easily 
checked.
\end{example}

\begin{example}
Given a finite set $X=\{x_1,\dots, x_n\}$, we can define biquandle structures 
on $X$ by encoding the operation tables of $\utr$ and $\otr$ as blocks in 
a matrix. For example, the set $X=\{x_1,x_2,x_3\}$ is a biquandle with 
operations defined by the operation tables
\[
\begin{array}{r|rrr}
\utr & x_1 & x_2 & x_3 \\ \hline
x_1 & x_1 & x_3 & x_2 \\
x_2 & x_3 & x_2 & x_1 \\
x_3 & x_2 & x_1 & x_3
\end{array}
\quad\mathrm{and}\quad
\begin{array}{r|rrr}
\utr & x_1 & x_2 & x_3 \\ \hline
x_1 & x_1 & x_1 & x_1 \\
x_2 & x_2 & x_2 & x_2 \\
x_3 & x_3 & x_3 & x_3
\end{array}
\]
which we abbreviate by dropping the ``$x$''s
to obtain the biquandle matrix
\[\left[\begin{array}{rrr|rrr}
1 & 3 & 2 & 1 & 1 & 1 \\
3 & 2 & 1 & 2 & 2 & 2 \\
2 & 1 & 3 & 3 & 3 & 3
\end{array}\right].\]
The left-hand block represents the operation $\utr$ while the right-hand block
represents the operation $\otr$; then for instance
we have $x_2\utr x_3=x_1$ and $x_3\otr x_1=x_3$, obtained by looking up the 
entries in row 2 column 3 and row 3 column 1 of the left and right blocks
respectively.
\end{example}

To construct knot invariants using finite biquandles, we assign a biquandle 
element to each semi-arc of a knot such that the pictured \textit{coloring 
condition:}
\[\includegraphics{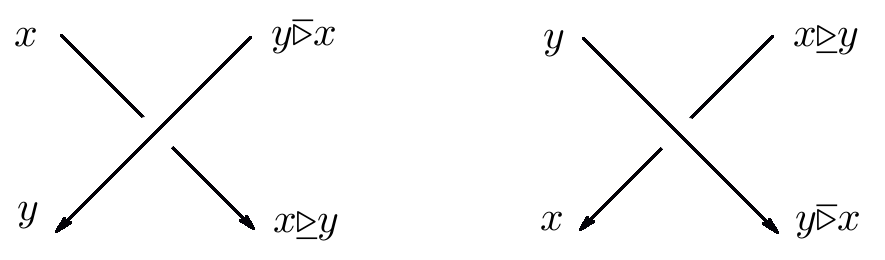}\]
is satisfied at every crossing.
This is sometimes called ``coloring'' the knot, and can be understood as a 
homomorphism from the \textit{fundamental biquandle} of the knot to the coloring biquandle. (The fundamental biquandle of a knot $K$, denoted $\mathcal{B}(K)$, is the set of equivalence classes of  biquandle words generated by
semiarcs in a diagram of $K$ modulo the crossing relations and  biquandle
operations). As in 
the picture, if a crossing is oriented downward as shown, the 
colorings of the two left-hand semi-arcs together with the two operators 
determine the colorings of the two right-hand semi-arcs.
In particular, we may interpret $x \utr y$ as $x$ after going under $y$, and 
$x \otr y$ as $x$ after going over $y$. Given a particular coloring biquandle 
and a particular knot, each crossing yields a constraint on the possible 
colorings of the knot with that biquandle.

Indeed, the biquandle axioms are chosen precisely to guarantee that for 
any biquandle coloring of a knot or link diagram before a Reidemeister
move, there is a unique corresponding coloring after the move:
\[\includegraphics{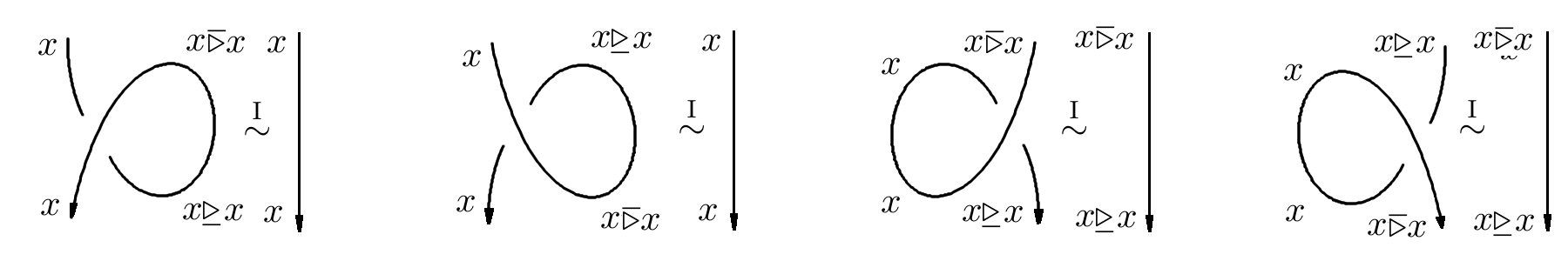}\]
\[\scalebox{0.9}{\includegraphics{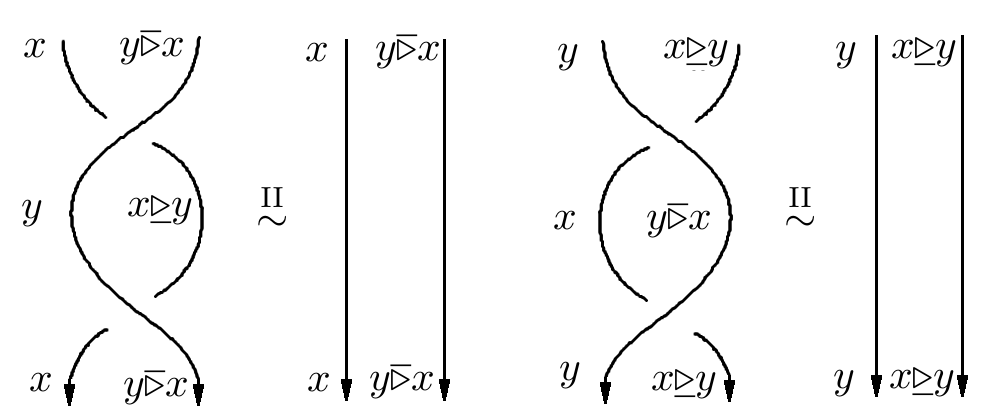}\
\includegraphics{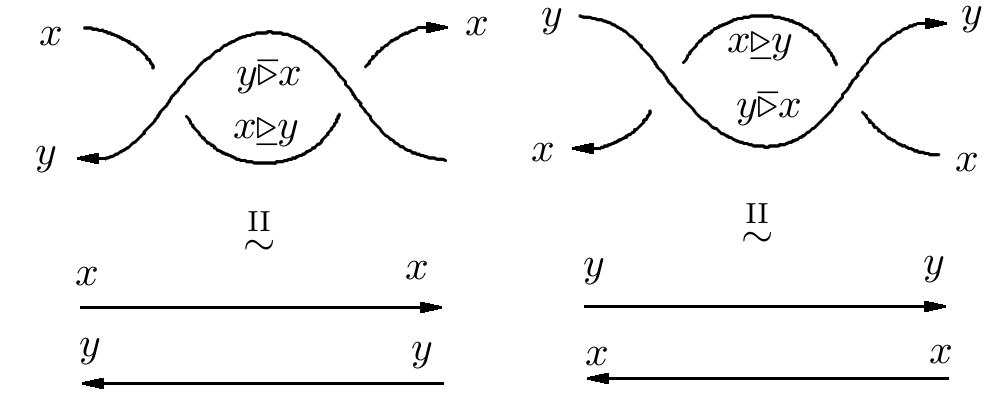}}\]
\[\includegraphics{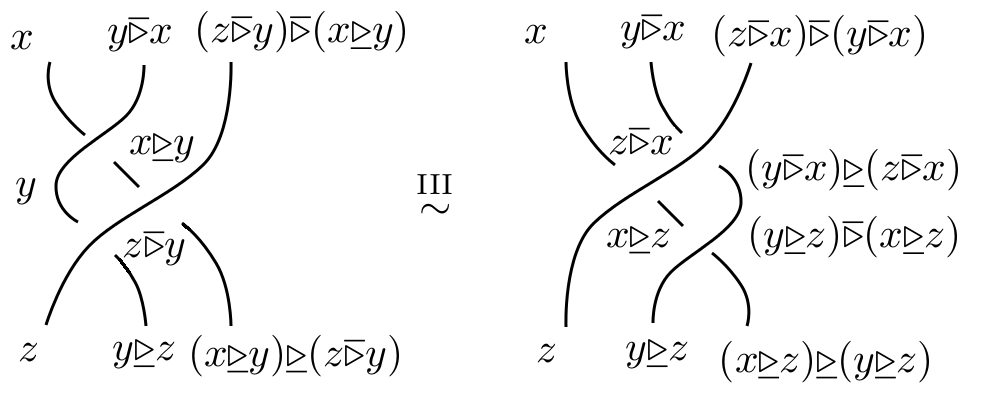}\]

Axiom 1 ensures that the semiarc created by a Reidemeister I move has a 
well-defined coloring, and Axiom 2 ensures a 1-1 correspondence between 
colorings before and after the move. The three invertibility conditions of 
Axiom 2 ensure 1-1 correspondences in different orientations of Reidemeister 
II moves. The exchange laws from Axiom 3 follow from the boundary conditions 
and multiple colorings in a Reidemeister III move. Note that Reidemeister I 
and II moves allow us to move between all the forms of the Reidemeister III 
move, permitting us to consider only the case with all three crossings positive.

Note that it is common elsewhere in the literature to define biquandle 
operations with the inbound oriented semiarcs operating on each other to 
produce the 
outbound oriented semiarcs, which we might call ``top-down'' operations; 
however, the 
``sideways'' operations are historically first (see \cite{FRS}) and have the 
advantages of resulting in more symmetric axioms and making biquandle 
homology much simpler.

We want to use biquandles to define an invariant based on assigning elements of 
a biquandle to knot semi-arcs.
Given a particular biquandle $X$, we have a counting invariant, 
which is the number of ways of assigning biquandle elements to a given knot. 

\begin{definition}
Let $X$ be a finite biquandle and $K$ an oriented knot. The \textit{biquandle
counting invariant} is the number of biquandle colorings of $K$ by $X$,
\[\Phi_X^{\mathbb{Z}}(L)=|\mathrm{Hom}(\mathcal{B}(K),X)|.\]
\end{definition}
As outlined above, every valid biquandle
labeling of a knot diagram before a Reidemeister move corresponds to a unique
valid biquandle labeling of the diagram after the move. This ensures that the above is indeed a knot invariant. 
\begin{theorem}
For any finite biquandle $X$, the corresponding biquandle counting invariant is knot invariant.
\end{theorem}
\section{\large\textbf{Parity Biquandles}}\label{PB}

To refine the biquandle counting invariant, we use an extension of the 
biquandle, the \textit{parity biquandle.} First defined in \cite{KK}, 
parity biquandles are similar to biquandles, but 
with four operations instead of just two: $\otr^0, \otr^1, \utr^0, \utr^1$, 
and some additional restrictions.
\begin{definition}
A \textit{parity biquandle} is a set $X$ along with four operations: 
$\otr^0. \otr^1, \utr^0, \utr^1$, all maps $X \times X \mapsto X \times X$, 
such that:
\begin{enumerate}
\item[(i)] $X$ along with the two operations $\otr^0$ and $\utr^0$ is a 
biquandle ($X$ along with $\otr^1$ and $\utr^1$ need not be).
\item[(ii)] We have right invertibility of both odd maps and pairwise 
invertibility of those maps, i.e.\ the maps 
$\alpha_y^1:x \mapsto x\otr^1 y, \beta_y^1:x \mapsto x\utr^1 y$, and 
$S: (x,y) \mapsto (y\otr^1 x, x\utr^1 y)$ are all invertible.
\item[(iii)] We have the \textit{mixed exchange laws}: 
\begin{align*}
(z \otr^a y) \otr^b (x \otr^c y) &= (z \otr^b x) \otr^a (y \utr^c x)\\
(x \otr^a y) \utr^b (z \otr^c y) &= (x \utr^b z) \otr^a (y \utr^c z)\\
(y \utr^a x) \utr^b (z \otr^c x) &= (y \utr^b z) \utr^a (x \utr^c z)
\end{align*}
for $(a,b,c) \in \{(0,1,1),(1,0,1),(1,1,0)\}$ 
(The case where $a=b=c=0$ must hold, but is enforced by condition 1).
\end{enumerate}
Note that biquandles are the special case of parity biquandles where 
$x \utr^1 y = x \utr^0 y$, and $x \otr^1 y = x \otr^0 y$. 
\end{definition}

For coloring virtual knots with parity biquandles, we will use the following
definition (see also \cite{KK}):

\begin{definition}
Let $K$ be a virtual knot diagram and let $C$ be a classical crossing in $K$.
We will say $C$ has \textit{parity} $0$ if the number of classical over and 
under crossings encountered traveling along $K$ between the under and over 
instances of $C$ is even, and we will say $C$ has \textit{parity} $1$ if the 
number of classical over and under crossings encountered traveling along $K$ 
between the under and over instances of $C$ is odd.
\end{definition}

Parity biquandles can be used to capture additional structure in virtual knots 
by using operators based on the parity of each crossing, with the $1$ 
superscript for odd crossings and the $0$ superscript for even ones. 
We will occasionally find it convenient to decorate each 
crossing with a $0$ or a $1$ to explicitly indicate its parity.

Given a parity biquandle $X$, we color the semiarcs of a virtual knot diagram
$D$ with elements of $X$, with constraints generated by the even operations 
$\utr^0,\otr^0$ at even crossings and the odd 
operations $\utr^1,\otr^1$ at the odd crossings.
\[\includegraphics{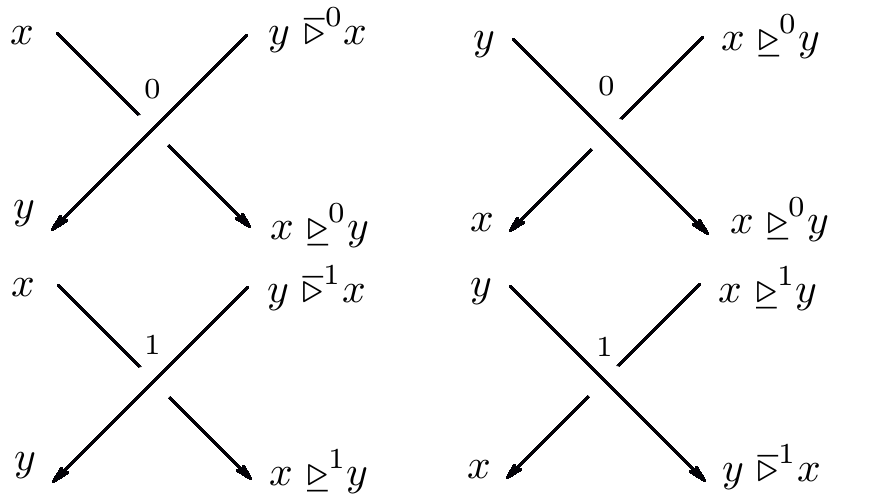}\]


\begin{remark}
Note that parity biquandle colorings only differ from biquandle colorings in 
non-classical virtual knots, since classical knots have only even 
crossings. This makes the odd operators irrelevant for classical knots.
\end{remark}

As with a biquandle, given a parity biquandle $X$, we can find the number of 
ways of assigning elements of $X$ to semi-arcs of a knot $K$, or ``coloring'' 
$K$, respecting the parity biquandle's relations. As before, these colorings can be defined as homomorphisms from $\mathcal{PB}(K)$, the \textit{fundamental parity biquandle} of
$K$, to the coloring biquandle $X$. (The fundamental parity biquandle is the set of equivalence classes of parity biquandle words generated by
semiarcs in a diagram of $K$ modulo the crossing relations and parity biquandle
operations.)
\begin{definition}
The \textit{parity biquandle counting invariant} is defined by
\[\Phi_X^{\mathbb{Z}}(K)=|\mathrm{Hom}(\mathcal{PB}(K),X)|\]
\end{definition}

To show 
that this is indeed an invariant, we must show that a single coloring before 
a Reidemeister move implies a unique coloring after the move. To do this, we 
look at the moves one at a time:

For type I moves, the crossing involved in the move is always even, so the 
biquandle condition on the even-crossing operators forces invariance.

For type II moves, we note that the signs of both crossings must be the same 
(by inspection of Gauss diagram). If both are even, invariance is forced by 
the biquandle condition on the even-crossing operators. If both are odd, 
invariance is forced by the invertibility rules from the definition.

For type III moves, the biquandle condition forces the all-even exchange laws. 
This forces invariance under the all-even Reidemeister III move. To show that 
the mixed exchange laws capture all the other cases, we need a lemma:

\begin{lemma}
In a Reidemeister III move, either all three crossings involved are even, or 
two are odd and one is even.
\label{lemma:r3par}
\end{lemma}

\begin{proof} 
Consider the possible Gauss diagrams that might start a type III move. 
Below we have diagrams of all the possible starting positions for the crossings involved in the move. (For more about Gauss diagrams, see for instance 
\cite{GPV}).
We can divide the crossing labels not involved in the move into three sections
(i.e., the dotted portions of the outer circle), based on their 
``minor segment'' in the diagrams below. 
\[\scalebox{0.5}{\includegraphics{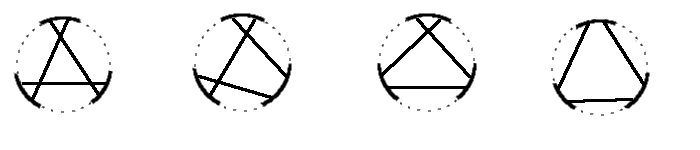}}\]
This total will be even, so either all three partitions are even or two are odd and one is even. Each minor segment's 
parity determines the parity of the crossing corresponding with that segment's 
chord. So the crossing parities must follow the same pattern: either all are even, or 
one is even and two odd.
\end{proof}

Given Lemma \ref{lemma:r3par}, it is clear that the mixed exchange laws 
encompass all of the remaining cases. 

We can construct examples to show that all of the cases not ruled out by 
Lemma \ref{lemma:r3par} are in fact possible. This is why all three cases of 
the mixed exchange laws are necessary. The figure below shows three knots 
where a Reidemeister III move is possible, each with the relevant crossing 
parities labeled. 
\[\includegraphics{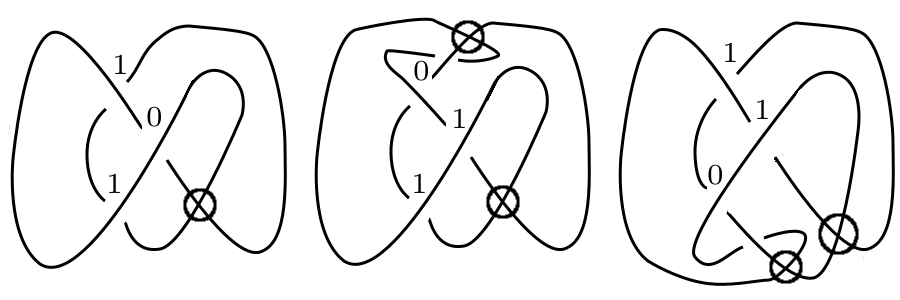}\]

We now have our desired result.
\begin{theorem}
For any finite parity biquandle $X$, the corresponding parity biquandle counting invariant is a knot invariant.
\end{theorem}

\begin{example}
As an example of a parity biquandle structure, we construct a family of 
parity biquandles by extending the Alexander biquandle structure. 
An \textit{Alexander parity biquandle} is a module $X$ over the 
ring $\Lambda=\mathbb{Z}[t^{\pm 1}, s^{\pm 1}, b^{\pm 1},a^{\pm 1}]$ of four-variable
Laurent polynomials (i.e.\ with all four variables invertible in the 
coefficient ring). The following constraints must also be satisfied:

\begin{align*}
(a^{-1}-b)^2+(s^{-1}-t)(b-a^{-1})&=0\\
(a^{-1}-b)(b-t)&=0\\
(a^{-1}-b)(s^{-1}-a^{-1})&=0
\end{align*}

The operators are defined as:
\begin{align*}
x \utr^0 y &= tx + (s^{-1}-t)y\\
x \otr^0 y &= s^{-1}y\\
x \utr^1 y &= bx + (a^{-1}-b)y\\
x \otr^1 y &= a^{-1}y
\end{align*}

The first two parity biquandle axioms follow from the definition and the 
(mixed) exchange laws follow from the constraints.
\end{example}

\begin{example}
Given a finite set $X=\{x_1,\dots,x_n\}$ we can define parity biquandle 
structures on $X$ with a $2n\times 2n$ block matrix $M$ encoding the operation
tables of the even and odd operations such that 
$x_i \utr^{\epsilon} x_j=M_{i,j+\epsilon n}$ and 
$x_i \otr^{\epsilon} x_j=M_{i+n,j+\epsilon n}$, where $\epsilon \in \{0,1\}$ and 
$M_{i,j}$ is the entry of $M$ in row $i$ column $j$. That is, we will encode 
the operations tables as a block matrix whose blocks are the operation tables 
of the operations arranged  as
$\left[\begin{array}{c|c}
\utr^0 & \utr^1 \\ \hline
\otr^0 & \otr^1 \\
\end{array}\right].$
For example, the set $X=\{x_1,x_2,x_3\}$ has parity biquandle structures 
including
\[\left[\begin{array}{rrr|rrr}
3 & 1 & 3 & 3 & 1 & 3 \\
2 & 2 & 2 & 2 & 2 & 2 \\
1 & 3 & 1 & 1 & 3 & 1 \\ \hline
1 & 3 & 1 & 3 & 3 & 3 \\
2 & 2 & 2 & 2 & 2 & 2 \\
3 & 1 & 3 & 1 & 1 & 1
\end{array}\right].\]
Then in this case, we have $3\utr^01=1$ and $1\otr^1 2=3$.
\end{example}

\section{\large\textbf{Parity Cocycle Enhancements}}\label{E}

We begin this section with a brief review of biquandle homology; see 
\cite{CES,CN,CEGN}, etc., for more.

Let $X$ be a finite biquandle and $A$ an abelian group. Define 
$C_n(X;A)=A[X^n]$, the free $A$-module on ordered $n$-tuples of elements of $X$.
For each $n=1,2,3,\dots$, define $\partial_n:C_n(X)\to C_{n-1}(X)$ by setting
\[\partial_n(x_1,\dots,x_n)
=\sum_{k=1}^n (-1)^k(\partial^0_{n,k}(x_1,\dots, x_n)-\partial^1_{n,k}(x_1,\dots,x_n))\]
where
\begin{eqnarray*}
\partial^0_{n,k}(x_1,\dots,x_n)& = & (x_1,\dots,x_{k-1},\ x_{k+1},\dots, x_n) \\
\partial^1_{n,k}(x_1,\dots,x_n)& = & (x_1\utr x_k,\dots, x_{k-1}\utr x_k,\ x_{k+1}\otr x_k,\dots, x_n\otr x_k)
\end{eqnarray*}
and extending linearly.

Then (see \cite{CEGN} for example) $\partial$ is a boundary map; the 
$A$-modules $H_n(X;A)=\mathrm{Ker}(\partial_n)/\mathrm{Im}(\partial_{n+1})$ and
$H^n(X;A)=\mathrm{Ker}(d^{n+1})/\mathrm{Im}(d^n)$ (where 
$d^n(f(x))=f(\partial_n(x))$ for any $f:C_n(X)\to A$, i.e.\ for any 
$f \in C^n$) are the \textit{$n$th birack 
homology and cohomology modules with coefficients in $A$} respectively. (A 
\textit{birack} is like a biquandle but without the conditions resulting
from the type I move replaced with conditions arising from the framed type I 
move; see for instance \cite{NR} for more).
 A birack
cocycle which evaluates to zero on all \textit{degenerate chains} ($A$-linear
combinations of generators $(x_1,\dots, x_n)$ with $x_k=x_{k+1}$ for some
$k=1,\dots, n-1$) is a \textit{reduced cocycle}.

In \cite{CES,CN} and more, biquandle $2$-cocycles are used to define
enhancements of the biquandle counting invariant for finite biquandles. 
Specifically, for any $X$-colored knot or link diagram, a reduced
2-cocycle $\phi$
is evaluated on the colors at each crossing; the algebraic sum of these
cocycle values (i.e.\ $+\phi(x,y)$ at positive crossings and $-\phi(x,y)$
at negative crossings, where $x,y$ are the under- and over-crossing colors
at the crossing), known as a \textit{Boltzmann weight}, is then invariant 
under $X$-labeled Reidemeister moves. Then the multiset
of Boltzmann weights over the set of $X$-colorings of $L$ is an enhanced
invariant of $L$ with cardinality equal to the $X$-counting invariant. It is
common to rewrite this multiset as a polynomial by taking the 
generating function of the multiset, i.e.\ converting multiplicities to 
integer coefficients and multiset elements to exponents of a dummy variable $u$.

To incorporate parity, we define the notion of a \textit{parity enhanced 
biquandle $2$--cocycle}, which is a reduced $2$--cocycle $\phi^0\in C^2(X;A)$ 
paired with another function $\phi^1:X\times X\to A$ satisfying certain
compatibility conditions with $\phi^0$. Essentially, the idea is to  evaluate
$\phi^0$ at even crossings and $\phi^1$ at odd crossings. 
\[\includegraphics{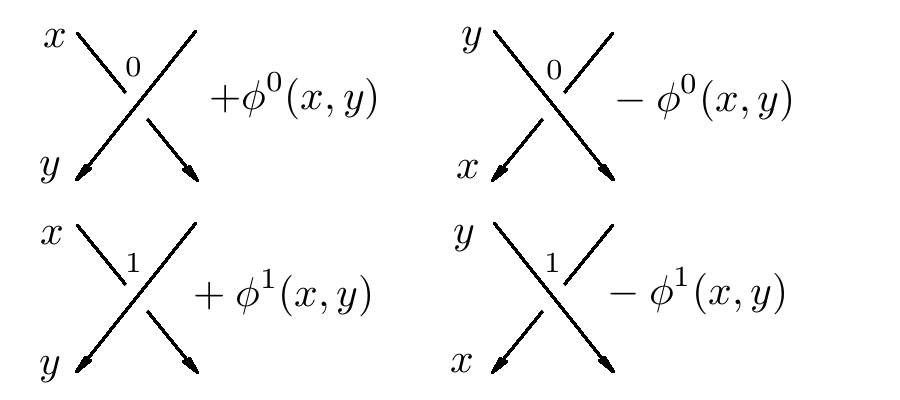}\]
Looking at the Reidemeister III move, we have
\[\includegraphics{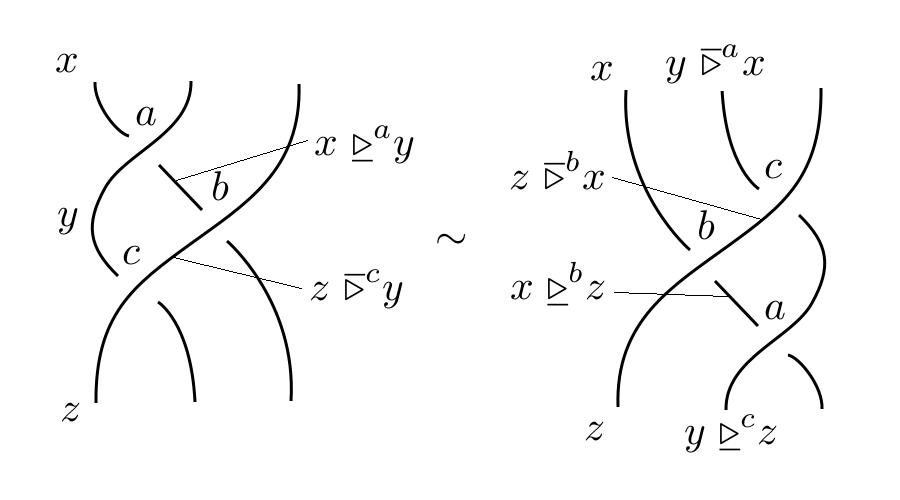}\]
and hence we need 
\[\phi^a(x,y)+\phi^b(x\utr^a y,x\otr^c y)+\phi^c(y,z)=
\phi^a(x\utr^bz,y\utr^cz)+\phi^b(x,z)+\phi^c(y\otr^a x,z\otr^b x)\]
for triples $(a,b,c)\in\{(0,0,0),(1,1,0),(1,0,1),(0,1,1)\}$. 

Analogously to \cite{CN}, we can define two forms of compatibility between
$\phi^0$ and $\phi^1$:

\begin{definition}
Let $X$ be a finite biquandle, $A$ an abelian group and 
$\phi^0,\phi^1:A[X^2]\to A$ linear maps. We say $\phi^0$ and $\phi^1$
are \textit{compatible} if for all $x,y,z\in X$ and for all
$(a,b,c)\in\{(1,1,0),(1,0,1),(0,1,1)\}$ we have
\[\phi^a(x,y)+\phi^b(x\utr^a y,x\otr^c y)+\phi^c(y,z)=
\phi^a(x\utr^bz,y\utr^cz)+\phi^b(x,z)+\phi^c(y\otr^a x,z\otr^b x).\]
We say $\phi^0$ and $\phi^1$ are \textit{strongly compatible} if 
$\phi^0$ and $\phi^1$ are compatible and we additionally have
\begin{eqnarray*}
\phi^0(y,z) &= & \phi^0(y\otr^1 x,z\otr^1 x),\\
\phi^0(x,z)&= &\phi^0(x\utr^1 y,z\otr^1 y)\mathrm{,\ and} \\
\phi^0(x,y) &= & \phi^0(x\utr^1z,y\utr^1z)
\end{eqnarray*}
for all $x,y,z\in X$.
\end{definition}

The compatibility condition together with the 2-cocycle condition for $\phi^0$
guarantees that the sum $\phi^j(x,y)$ of contributions from $\phi^0$ and 
$\phi^1$ at even and odd crossings is not changed by Reidemeister III moves.
The strong compatibility condition guarantees that the separate
contributions from even and odd crossings are preserved by type III moves. 
The contribution rules guarantee that (even or odd) type II moves do not
change the total contribution, and the reduced condition for $\phi^0$ 
guarantees
that type I moves do not change the overall sum of crossing weights. Thus, 
we have:

\begin{definition}
Let $X$ be a finite parity biquandle, $K$ a virtual knot, $A$ an 
abelian group and $\phi^0\in H^2(X;A)$ a reduced biquandle $2$-cocycle, and
$\phi^1:A[X^2]\to A$ a linear map compatible with $\phi^0$. 
For each $f\in\mathrm{Hom}(\mathcal{PB}(K),X)$, the \textit{parity Boltzmann 
weight} of $f$ is the sum
\[BW(f)=\sum_{c\ \mathrm{crossings}} \sigma(c)\phi^{\epsilon(c)}(x_c,y_c)\]
where $\sigma(c)=\pm 1$ is the sign of the crossing, $\epsilon(c)\in\{0,1\}$ 
is the 
parity of the crossing, and $(x_c,y_c)$ are the left side under- and over-crossing
labels. If $\phi^0$ and $\phi^1$ are strongly compatible, the \textit{strong 
parity Boltzmann weight} of $f$ is
\[SBW(f)=(SBW(f)_0,SBW(f)_1)
=\left(\sum_{\mathrm{even\ crossings}} \sigma(c)\phi^0(x_c,y_c),\
\sum_{\mathrm{odd\ crossings}} \sigma(c)\phi^1(x_c,y_c)\right).\]
Then the \textit{parity enhanced biquandle cocycle multiset} of $K$ is
the multiset
\[\Phi^{\phi,M}_X(K)=\{BW(f)\ |\ f\in\mathrm{Hom}(\mathcal{PB}(K),X)\}\]
or
\[\Phi^{\phi,sM}_X(K)=\{SBW(f)\ |\ f\in\mathrm{Hom}(\mathcal{PB}(K),X)\}\]
in the strongly compatible case.
The \textit{parity enhanced biquandle cocycle polynomial} of $K$ is
\[\Phi^{\phi}_X(K)=\sum_{f\in\mathrm{Hom}(\mathcal{PB}(K),X)} u^{BW(f)}\] or
\[\Phi^{\phi,s}_X(K)=\sum_{f\in\mathrm{Hom}(\mathcal{PB}(K),X)} u^{SBW(f)_0}v^{SBW(f)_1}\]
in the strongly compatible case.
\end{definition}

By construction, we have our main result:
\begin{proposition}
Let $X$ be a parity biquandle and $\phi$ a parity-enhanced cocycle. If
two virtual knots $K$ and $K'$ are related by virtual Reidemeister moves, then
\[\Phi^{\phi,M}_X(K)=\Phi^{\phi,M}_X(K') \quad \mathrm{and}\quad 
\Phi^{\phi}_X(K)=\Phi^{\phi}_X(K').\]
If $\phi^0$ and $\phi^1$ are strongly compatible, we have
\[\Phi^{\phi,sM}_X(K)=\Phi^{\phi,sM}_X(K') \quad \mathrm{and}\quad 
\Phi^{\phi,s}_X(K)=\Phi^{\phi,s}_X(K').\]
Therefore, all four are knot invariants.
\end{proposition}

We can conveniently
specify any pair $\phi^0,\phi^1:X\times X\to A$ with an $n\times 2n$
block matrix with entries in $A$ representing the coefficients of the 
characteristic maps $\chi_{x_i,x_j}$. For instance, if $X=\{x_1,x_2\}$ then
we use the matrix
\[\left[\begin{array}{rr|rr}
2 & 1 & 0 & 1 \\
0 & -1 & 1 & -2
\end{array}\right]\]
to indicate the maps $\phi^0=2\chi_{(x_1,x_1)}+\chi_{x_1,x_2}-\chi_{(x_2,x_2)}$
and $\phi^1=\chi_{(x_1,x_2)}+\chi_{x_2,x_1}-2\chi_{(x_2,x_2)}$.

\begin{remark}
If $\phi^0$ is a biquandle 2-cocycle and $\phi^1$ and $\psi^1$ are both strongly
compatible with $\phi^0$, then we note that $\phi^1+\psi^1$ and 
$\alpha\phi^1$ for $\alpha\in A$ are also strongly compatible with 
$\phi^0$. In particular, for each biquandle 2-cocycle, the set of strongly 
compatible maps has the structure of an $A$-module.
\end{remark}

Our first example illustrates the computation of the invariant and demonstrates
that the parity cocycle enhancement provides more information than the 
corresponding unenhanced biquandle counting invariant.

\begin{example}
Consider the parity biquandle $X$ with elements \{1,2,3\} and operation matrix
\[\left[\begin{array}{rrr|rrr}
3 & 1 & 3 & 3 & 1 & 3 \\
2 & 2 & 2 & 2 & 2 & 2 \\
1 & 3 & 1 & 1 & 3 & 1 \\ \hline
1 & 3 & 1 & 3 & 3 & 3 \\
2 & 2 & 2 & 2 & 2 & 2 \\
3 & 1 & 3 & 1 & 1 & 1
\end{array}\right].\]
Our \texttt{python} searches reveal that $X$ has strongly compatible
parity enhanced cocycles over $A=\mathbb{Z}_5$, including
\[\phi=\left[\begin{array}{rrr|rrr}
0 & 0 & 0 & 0 & 2 & 0 \\
2 & 0 & 2 & 2 & 3 & 2 \\
0 & 0 & 0 & 0 & 2 & 0
\end{array}\right].\]
Then the virtual trefoil knot $2.1$ has three $X$-colorings with Boltzmann 
weights as depicted: 
\[\includegraphics{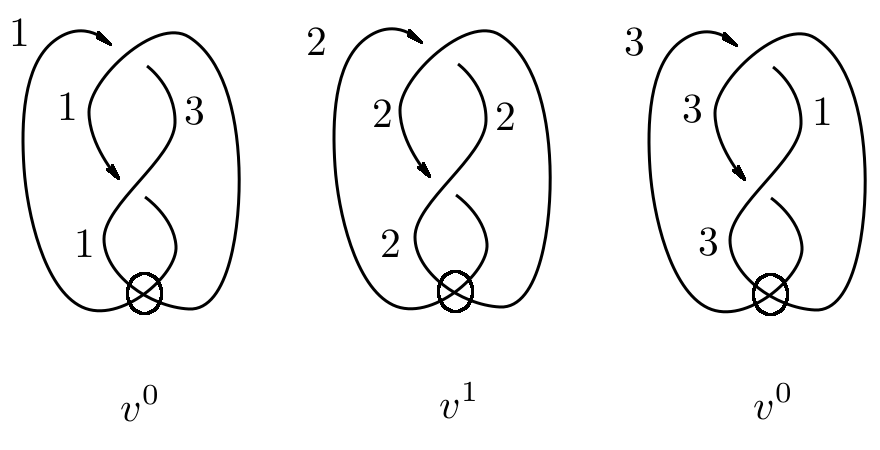}\]
yielding a parity-enhanced biquandle cocycle invariant value of 
$\Phi^{\phi}_X(2.1)=2+v$, distinguishing it from the unknot with 
$\Phi^{\phi}_X(0.1)=3u^0=3$.
On the other hand, the corresponding non-parity biquandle cocycle invariant 
(treating all crossings as even) has value $3u^0=3$ for both $2.1$ and 
the unknot.
\end{example}

For our next example, we chose a four-element biquandle and strongly
compatible parity enhanced cocycle over $A=\mathbb{Z}_3$ and computed the
invariant for all prime virtual knots with up to four classical crossings
as listed in the knot atlas \cite{KA}.

\begin{example}
Let $X$ be the parity biquandle with operation matrix
\[\left[\begin{array}{rrrr|rrrr}
3 & 4 & 2 & 1 & 3 & 4 & 2 & 1\\
1 & 2 & 4 & 3 & 1 & 2 & 4 & 3\\
4 & 3 & 1 & 2 & 4 & 3 & 1 & 2\\
2 & 1 & 3 & 4 & 2 & 1 & 3 & 4\\ \hline
1 & 3 & 1 & 3 & 1 & 3 & 1 & 3\\
2 & 4 & 2 & 4 & 2 & 4 & 2 & 4\\
3 & 1 & 3 & 1 & 3 & 1 & 3 & 1\\
4 & 2 & 4 & 2 & 4 & 2 & 4 & 2 
\end{array}\right]\]
and $\phi$ the strongly compatible parity enhanced cocycle over 
$\mathbb{Z}_3$ with matrix
\[\left[\begin{array}{rrrr|rrrr}
0 & 2 & 2 & 1 & 1 & 1 & 1 & 1\\
2 & 0 & 1 & 2 & 1 & 1 & 1 & 1\\
2 & 1 & 0 & 2 & 1 & 1 & 1 & 1\\
1 & 2 & 2 & 0 & 1 & 1 & 1 & 1
\end{array}\right].\]
Then our \texttt{python} computations reveal values of the two-variable
 parity cocycle invariant $\Phi^{\phi,s}_X$ for 
the prime virtual knots with up to 4 classical crossings as listed in the table
(numbered as in the knot atlas \cite{KA}).
The double lines divide the table by parity biquandle
counting invariant value $\Phi^{\mathbb{Z}}_X$ and the single
lines divide the table by single-variable parity cocycle enhancement value
$\Phi^{\phi}_X$.
\[
\begin{array}{r|l}
\Phi^{\phi}_X(K) & K \\ \hline 
4 & 3.1, 3.5, 3.6, 3.7, 4.2, 4.6, 4.8, 4.12, 4.13, 4.17, 4.19, 4.26, 4.46, 4.47, 4.51, 4.55, 4.56, 4.75, 4.76, \\
 & 4.77, 4.86, 4.93, 4.96, 4.97, 4.99, 4.102, 4.103, 4.105, 4.106, 4.108\\ 
4u^2v & 4.29, 4.37, 4.61, 4.69 \\ \hline
4u & 4.36, 4.68 \\
4v & 3.2, 3.3, 3.4, 4.4, 4.5, 4.11, 4.18, 4.27, 4.44, 4.45, 4.49, 4.54, 4.74, 4.81, 4.82, 4.83, 4.87, 4.92,\\
 &  4.94, 4.95, 4.101 \\ \hline
4u^2 & 4.31, 4.41, 4.57, 4.65, 4.70 \\
4uv & 4.34, 4.40, 4.60, 4.64 \\
4v^2 & 2.1, 4.1, 4.3, 4.7, 4.25, 4.28, 4.43, 4.53, 4.73, 4.80, 4.84, 4.88, 4.91, 4.100, 4.104 \\ \hline\hline

8 & 4.10, 4.16, 4.21, 4.23, 4.24, 4.50, 4.79 \\ \hline
8v & 4.9, 4.14, 4.15, 4.20, 4.22, 4.48, 4.52, 4.78 \\ \hline
4u^2+4u & 4.32, 4.35, 4.42, 4.58, 4.59, 4.66, 4.67, 4.71, 4.72 \\ \hline
4u^2v+4uv & 4.30, 4.33, 4.38, 4.39, 4.62, 4.63 \\ \hline \hline

16 & 4.90, 4.98 \\ \hline
4u^2+12 & 4.89 \\ \hline
4u^2+4u+8 & 4.107 \\ \hline
8u^2+4u+4 & 4.85 \\
\end{array}\]
\end{example}

\begin{example}
For any parity biquandle $X$ , the maps $\phi^0(x,y)=0$ 
and $\phi^1(x,y)=1$ for all $x,y\in X$ define a strongly compatible 
parity-enhanced cocycle. The resulting invariant $\Phi^{\phi,s}_X(X)$ has value
\[\Phi^{\phi,s}_X(K)=\Phi^{\mathbb{Z}}_X(K)v^{OW(K)}\]
where 
\[OW(K)=\sum_{c\ \mathrm{odd\ crossing}}\epsilon(c)\]
is the \textit{odd writhe} of $K$, the sum of crossing signs at odd crossings.
In particular, if $\Phi^{\phi}_X(K)\ne\Phi^{\mathbb{Z}}_X(K)$ then $K$ must be
non-classical. Moreover, this example shows that cohomologous cocycles
$\phi^0$ and $\psi^0$ need not define the same parity-enhanced invariant,
unlike the traditional case.
\end{example}

\section{\large\textbf{Questions}}\label{Q}

In what situations are two virtual knots distinguished by parity biquandle invariants but 
not by the corresponding biquandle invariants? In other words, what conditions 
are sufficient for the parity biquandle invariants to be stronger than
then their non-parity counterparts?

We would like to express the parity Boltzmann weight invariant in the language 
of cohomology. So far we don't have a satisfactory sense of ``parity'' for 
elements of
$C_3$, representing knotted surfaces, which leads to a useful boundary map. 
Specifically we would like a boundary map which yields 2-crossings of the 
correct parities. 
We also wonder about possible parity in $C_1$.

A related question is what the relationship is between $\phi^0$ and 
$\phi^1$ in general? As we've seen, the set of $\phi^1$ strongly
compatible with a given $\phi^0$ forms an $A$-module; what is the relationship
of these modules with $C^2(X;A)$? Is there some deeper homology theory, perhaps
with some ``parity grading'', from which the compatibility conditions emerge 
naturally?

\bibliography{sn-ls-rev}{}
\bibliographystyle{abbrv}

\bigskip

\noindent
\textsc{Department of Mathematics \\
3225 West Foster Avenue\\
Chicago, IL 60625-4895}

\medskip

\noindent
\textsc{Department of Mathematical Sciences \\
Claremont McKenna College \\
850 Columbia Ave. \\
Claremont, CA 91711}

\end{document}